\documentclass[11pt,a4paper,english,reqno]{amsart}
\usepackage{amsmath,amssymb,amsfonts,epsfig,mathrsfs}
\usepackage[T1]{fontenc}

\usepackage{color}
\usepackage{array}
\usepackage{amsthm}
\usepackage{amstext}
\usepackage{graphicx}
\usepackage{setspace}
\usepackage[margin=2.5cm]{geometry}
\usepackage{bbm}
\usepackage{color}
\usepackage{enumitem}
\allowdisplaybreaks[4]

\usepackage{pgfplots}

\usepackage{amscd,psfrag}
\usepackage{yhmath}
\usepackage[mathscr]{eucal}

\usepackage{nicefrac}

\usepackage{slashed}

\makeatletter
\pdfpageheight\paperheight
\pdfpagewidth\paperwidth

\usepackage{mathrsfs}

\usepackage{epstopdf}
\usepackage{chngcntr}
\counterwithin{figure}{section}
\usepackage{mathrsfs}

\usepackage{indentfirst}	

\usepackage[normalem]{ulem}
\theoremstyle{plain}

\newtheorem{definition}{Definition}[section]
\newtheorem{theorem}[definition]{Theorem}
\newtheorem*{theorem*}{Theorem}

\newtheorem*{remark*}{Remark}
\newtheorem*{sideremark*}{Side Remark}\newtheorem*{mt*}{Main Theorem}

\newtheorem*{claim*}{Claim}
\newtheorem*{q*}{Question}

\newtheorem*{corollary*}{Corollary}
\newtheorem*{proposition*}{Proposition}

\newcommand{\R}{\mathbb{R}}

\newcommand{\p}{\partial}
\newcommand{\e}{\epsilon}

\newcommand{\map}{\rightarrow}
\newcommand{\G}{\Gamma}

\newcommand{\essinf}{{\rm ess\, inf}}\newcommand{\esssup}{{\rm ess\, sup}}

\newcommand{\sing}{{\bf sing}_{\rm Cl}}

\newcommand{\mres}{\mathbin{\vrule height 1.6ex depth 0pt width
0.13ex\vrule height 0.13ex depth 0pt width 1.3ex}}

\allowdisplaybreaks[4]

\def\XXint#1#2#3{{\setbox0=\hbox{$#1{#2#3}{\int}$ }
\vcenter{\hbox{$#2#3$ }}\kern-.6\wd0}}

\numberwithin{equation}{section}
\numberwithin{figure}{section}

\title{Cartan--Whitney Presentation, Non-smooth Analysis and Smoothability of Manifolds: On a theorem of Kondo--Tanaka}

\author{Siran Li}
\address{Siran Li: Department of Mathematics, Rice University, MS 136
P.O. Box 1892, Houston, Texas, 77251-1892, USA\, $\bullet$ \,  Department of Mathematics, McGill University, Burnside Hall, 805 Sherbrooke Street West, Montreal, Quebec, H3A 0B9, Canada.}

\email{\texttt{Siran.Li@rice.edu}}


\date{\today}

\begin{document}

\maketitle

\begin{abstract}
Using tools and results from geometric measure theory, we give a simple new proof of the main result in \cite{kt} (Theorem 1.3 in K. Kondo and M. Tanaka, Approximation of Lipschitz Maps via Immersions and Differentiable Exotic Sphere Theorems, \textit{Nonlinear Anal.} \textbf{155} (2017), 219--249), as well as the  converse statement. It explores the connections between the theory of non-smooth analysis {\it \`{a} la} F.~H. Clarke and the existence of special systems of Whitney flat $1$-forms with Sobolev regularity on certain families of homology manifolds.
\end{abstract}

\section{Introduction}

The smoothability of topological manifolds has long been a question at the heart of differential and geometric topology: {\em Given a manifold with structures of weak regularity, e.g., a topological, homology, or Lipschitz manifold, does it admit a smooth structure?}

Foundational works on  smoothability of manifolds by Whitehead \cite{w'} and Cairns \cite{c} (also see Pugh \cite{p} for an alternative, modern proof), Stallings \cite{st}, Shikata \cite{sh1, sh2},  Moise \cite{m} and Kirby--Siebenmann \cite{ks}, among many others, provide deep insights into the structures of manifolds. As popularised by Gromov in \cite{g'}, they address the simple yet fundamental question: ``what is a manifold''. The geometrical and topological developments in this line culminate in the discovery of {\em exotic} ({\it i.e.}, homeomorphic but not diffeomorphic) structures; see Milnor \cite{milnor}, Freedman \cite{fr}, Donaldson--Sullivan \cite{dosu} and Gromov \cite{g}.

On the other hand, using the techniques from geometric measure theory, an analytic approach has been developed to tackle the smoothability problem. Sullivan \cite{s1, s2, s3} initiated the programme of detecting the smoothability of a Lipschitz manifold using the notion of a  ``measurable cotangent bundle''. Its sections $\vartheta$ are identified with {\em flat forms}, which were introduced by Whitney in his theory of geometric integration theory \cite{w}. Roughly speaking, $\vartheta$ is a local coframe with weak regularity and an essentially nondegenerate volume density, and the integration of $\vartheta$ along segments gives rise to a branched covering map $F_\vartheta$. Heinonen--Rickman \cite{hr} and Heinonen--Sullivan \cite{hs} proved that the local smoothability of a Lipschitz manifold is equivalent to that  the local degree of $F_\vartheta = 1$; furthermore, Heinonen--Keith \cite{hk} established its equivalence with the Sobolev regularity condition $\vartheta \in W^{1,2}_{\rm loc}$.

In a recent paper \cite{kt},  a brand-new perspective has been adopted by Kondo--Tanaka to approach the smoothability problem. It connects F.~H. Clarke's theory of non-smooth analysis (\cite{c1,c2}), originally developed for applications in optimisation and control theory, to the approximation of Lipschitz maps by diffeomorphisms. The Main Theorem $1.3$ in \cite{kt} is as follows:
\begin{theorem}\label{thm: old}
Let $M$ be an $n$-dimensional compact Riemannian manifold, and let $N$ be an $\nu$-dimensional Riemannian manifold, where $2 \leq n \leq \nu$. Then, a Lipschitz map $F: M \map N$ is approximable by smooth immersions if $\sing\,F = \emptyset$.
\end{theorem}

$\sing\,F$ denotes the singular set of $F$ in the sense of Clarke \cite{c1, c2}; see Definition \ref{def: singular set}. The rigorous definitions for relevant geometric-analytic notions shall be given in $\S \ref{sec: prelim}$. The proof in \cite{kt} may be viewed as an intricate generalisation of the classical arguments by Grove--Shiohama \cite{gs}.

In this note we present a simple, new proof of Theorem \ref{thm: main}, which also establishes its converse at the same strike. Our proof is based on the geometric measure theoretic studies on the smoothability problem (see  \cite{s1, s2, s3, hr, hs, hk}). In particular, we make crucial use of the results due to Heinonen--Keith \cite{hk}. We hope it may provide an avenue for further explorations on the linkages between Clarke's non-smooth analysis  \cite{c1, c2} and geometric measure theory.

 In summary, we shall prove:
\begin{theorem}\label{thm: main}
In the setting of Theorem \ref{thm: old}, the Lipschitz map $F: M \map N$ is approximable by smooth immersions if and only if $\sing\,F = \emptyset$.
\end{theorem}

Alternative analytic approaches, besides the geometric measure theoretic and non-smooth analytic ones, have also been developed to study the smoothability problem; see Ball--Zarnescu \cite{bz} and the references cited therein. This paper and the subsequent developments also address the applications of manifold smoothability theorems to the modelling and analysis of liquid crystals.

\section{Background}\label{sec: prelim}

In this section we briefly discuss some preliminary materials on non-smooth analysis and geometric measure theory. For comprehensive treatments, we refer to \cite{c1, c2, kt} on  the former topic and to \cite{f, w, dase} on the latter.

\smallskip
\noindent
{\bf Smoothability.} A topological manifold $M$ is said to possess a $\wp$-structure ($\wp \in \{$Lipschitz, $C^{k,\alpha}$, smooth=$C^\infty$, analytic...$\}$) if and only if there is an atlas $\{(U_\alpha, \phi_\alpha): \alpha \in \mathcal{I}\}$ of $M$ such that all the transition maps $$\phi_\alpha \circ \phi_\beta^{-1}: \R^n \supset \phi_\beta(U_\alpha \cap U_\beta) \longrightarrow \phi_\alpha(U_\alpha \cap U_\beta) \subset \R^n$$ have $\wp$-regularity. By definition, a topological manifold has a $C^0$-structure. A $\wp$-manifold $M$ is said to be {\em smoothable} if there is a sub-atlas with respect to which $M$ admits a $C^\infty$-structure.

\smallskip
\noindent
{\bf Cartan--Whitney presentation.} Let $\mathcal{O} \subset \R^\nu$ be an open subset. A $k$-form $\omega$ is said to be a {\em (Whitney) flat $k$-form} if and only if $\omega$ has measurable coefficients and
\begin{equation*}
\omega, \, d\omega \in L^\infty(\mathcal{O}).
\end{equation*}
Here $d\omega$ is understood in the weak ({\it i.e.}, distributional) sense.

In what follows the definition of {\em Cartan--Whitney presentations} will be given. Let us first recall a prototypical case: Consider an $n$-dimensional topological manifold $M$, and let $U$ be an open neighbourhood of some point $p \in M$ with a well-defined orientation. Then there exists a local coframe $\{\theta^1, \ldots, \theta^n\} \in \G(T^*U)$ such that the differential $n$-form $\theta^1 \wedge\ldots\wedge \theta^n$ agrees with the orientation and is nondegenerate: for a constant $c>0$, 
\begin{equation}\label{theta i}
\int_U \theta^1 \wedge \ldots \wedge \theta^n \,{\rm d}V_{g_M} \geq c\, {\rm Vol}_{g_M}(U) > 0.
\end{equation}
Here and throughout, for a fibre bundle $E$ over $M$, $\G(E)$ denotes the space of its sections. Also, ${\rm d}V_{g_M}$ is the volume measure induced by $g_M$.

The notion of Cartan--Whitney presentations serves as a generalisation of the $n$-form $\theta^1 \wedge\ldots\wedge \theta^n$ in the above. Let $X$ be a metric space that is a ``nice'' $n$-dimensional subset of $\R^\nu$, $\nu \geq n$. Let $U$ be an open neighbourhood of a fixed point $p$ in $X$. A {\em (local) Cartan--Whitney presentation near $p$} consists of an $n$-tuple $\rho=(\rho_1,\ldots,\rho_n)$ of flat $1$-forms defined in an $\R^\nu$-neighbourhood $\mathcal{O}$ of $p$, such that $\mathcal{O}\cap X \subset U$ and 
\begin{equation}\label{rho}
\star (\rho_1\wedge\ldots\wedge \rho_n) \geq c' >0\qquad \text{ \textit{a.e.} on } \mathcal{O} \cap X   
\end{equation}
for some constant $c'>0$.

What does it mean by ``nice'' for $X$? On one hand, in the above definition $\rho$ is defined on $\mathcal{O} \subset \R^\nu$, so we have to ensure that its local restrictions to $X$ make sense. On the other hand, $U \subset X$ needs to have a good sense of orientation, so that the Hodge-star in \eqref{rho} is well-defined. Indeed, by the work \cite{se} of Semmes, the following conditions ensure that $X$ is nice enough to make sense of the above definition of Cartan--Whitney presentation:
\begin{align}\label{X}
&\text{$X$ is a locally Ahlfors $n$-regular,}\nonumber\\
&\qquad\qquad \text{ locally linearly contractible homology $n$-manifold, } n \geq 2.
\end{align}
Here, $X$ is {\em locally Ahlfors $n$-regular} if and only if $X$ has Hausdorff dimension $n$, and for every compact $K \Subset X$ there exist numbers $r_K >0$, $C_K\geq 1$ such that $$C_K^{-1}r^n \leq \mathcal{H}^n(B(x,r)) \leq C_Kr^n$$ for each metric ball $B(x,r) \subset X$ with $x \in K$ and $r<r_K$. $X$ is {\rm locally linearly contractible} if and only if for every compact $K \Subset X$ there exist numbers $r_K'>0$, $C_K' \geq 1$  such that every metric ball $B(x,r)  \subset X$ with $x \in K$ and $r<r_K'$ contracts to a point inside $B(x, C_K'r)$. Finally, $X$ is {\em homology $n$-manifold} if and only if it is  separable, metrisable, locally compact, locally contractible, and that for each $x \in X$ the following identity on homology groups holds:
\begin{equation}\label{hmlg mfd}
H_\bullet (X, X\sim \{x\}; \mathbb{Z}) \cong H_\bullet (\R^n, \R^n \setminus \{0\}; \mathbb{Z}). 
\end{equation}
The quadruple $(C_K, r_K; C_K', r_K')$ is called the {\em local data} of $X$ on $K$. One may refer to $\S \S 1$--3 in Heinonen--Keith \cite{hk} for detailed discussions. The punchline is: a local Cartan--Whitney presentation $\rho$ can be defined on $X \subset \R^\nu$ with weak regularity as in \eqref{X}.

\smallskip
\noindent
{\bf Sobolev space.} Next let us define the Sobolev space $W^{1,2}$ on $X$ satisfying \eqref{X}: it is the norm completion of Lipschitz functions $\phi: X \map \R$ with respect to
\begin{equation*}
\|\phi\|_{1,2} := \|\phi\|_{L^2(X)} + \|{\rm ap} D\phi\|_{L^2(X)},
\end{equation*}
where the {\em approximate differential} ${\rm ap} D\phi$ is {\it a.e.} defined on $X$ as in $3.2.19$ of Federer \cite{f}.


\smallskip
\noindent
{\bf Measureble cotangent structure; the theorem of Heinonen--Keith on smoothability.} Let $X$ be as in \eqref{X}. A result due to Cheeger \cite{ch} implies that $X$ is $n$-rectifiable. For any flat $1$-form $\omega$ defined on a open subset $\mathcal{O} \subset \R^\nu$ such that $\mathcal{O} \cap X =: U$ is non-empty, one can define the restriction $\omega \mres U$ as a map from $U$ to $T^*_x U$. The space $T^*_x U$ is viewed as a measurable section of $T^*U \subset T^*\R^\nu$, {\it i.e.} the {\em measurable cotangent bundle}; see $\S 3.4$ in \cite{hk} and p.303, Theorem 9A in \cite{w}.

Furthermore, in the pioneering works \cite{s1, s2, s3} Sullivan introduced the notion of a {\em measurable cotangent structure}. It consists of a pair $(E,\iota)$, where $E$ is an oriented rank-$n$ Lipschitz vector bundle over $X$, and $\iota$ is a module map over ${\rm Lip}(X,\R)$ from Lipschitz sections of $E$ to flat $1$-forms on $X$, such that the following holds: If $\sigma_1, \ldots, \sigma_n: X \map E$ are Lipschitz sections such that $\sigma_1 \wedge \ldots\wedge \sigma_n$ determines the chosen orientation on $E$, then for every $\alpha$ (index of an oriented, trivialised atlas $\{U_\alpha\}$) the flat $n$-form $$\iota(\sigma_1)|_\alpha \wedge \ldots \wedge \iota(\sigma_n)|_\alpha = \tau_\alpha dx^1 \wedge \ldots \wedge dx^n$$ satisfies $$\essinf_K\, \tau_\alpha >0$$ in each compact subset $K \Subset U_\alpha$. In practice, one considers $E=T^*U_\alpha$ as in the above paragraph.

The main result in \cite{hk} can be summarised as follows:
\begin{theorem}[Theorems 1.1 and 1.2, \cite{hk}]\label{thm: hk}
\begin{enumerate}
\item
If $X \subset \R^\nu$ as in \eqref{X} admits a Cartan--Whitney presentation in $W^{1,2}_{\rm loc}(X)$, then it is locally bi-Lipschitz parametrised by $\R^n$.
\item
An oriented Lipschitz manifold is smoothable if and only if it admits a measurable cotangent structure with local frames in $W^{1,2}_{\rm loc}(X)$.
\end{enumerate}
\end{theorem}

\smallskip
\noindent
{\bf Non-smooth analysis.} 
Let $(X, d_X)$ and $(Y, d_Y)$ be metric spaces. For a function $\phi: X \map Y$, one can define its Lipschitz norm as usual with respect to the metrics $d_X$, $d_Y$. As an example, let $(M, g_M)$ and $(N, g_N)$ be $n$- and $\nu$-dimensional Riemannian manifolds, and denote by $dist_M$ and $dist_N$ the Riemannian distance functions on $M$, $N$ induced by the Riemannian metrics, respectively. The Lipschitz norm of $\phi: M \map N$ is 
\begin{equation*}
\|\phi\|_{{\rm Lip}(M,N)} := \sup_{x \neq y, \, x, y \in M}\, \frac{{\rm dist}_N\big(\phi(x),\phi(y)\big)}{{\rm dist}_M (x,y)}.
\end{equation*}
By Rademacher's theorem, at ${\rm d}V_{g_M}$-\textit{a.e.} point $x\in M$ a Lipschitz function $\phi:M\map N$ has well-defined differential $d_x\phi:T_xM \map T_{\phi(x)}N$. 

For $\phi: (X, d_X) \map (Y, d_Y)$ as above, its {\em generalised differential} is the set-valued function:
\begin{equation*}
\p \phi (x) := {\rm conv} \Big(\Big\{ \lim_{i \map \infty} d_{x_i}\phi:\, d_{x_i}\phi \text{ exists  and}  \, dist_M (x_i,x) \map 0 \Big\}\Big).
\end{equation*}
Here $conv$ denotes the convex hull. Note that for any $x\in X$, each element of $\mathfrak{m} \in \p \phi (x)$ can be identified with a matrix; thus we may introduce the following
\begin{definition}\label{def: singular set}
The singular set of $\phi$ {\em \`{a} la} Clarke (\cite{c1,c2}) is
\begin{equation*}
\sing\,\phi := \Big\{ x\in X: \text{ there exists }\mathfrak{m} \in \p\phi(x) \text{ that is not of the maximal rank}\Big\}.
\end{equation*}
\end{definition}

A function $\phi: X \map Y$ is said to be {\em approximable by smooth immersions} if for any $\e>0$ there exists a smooth immersion $\iota_\e: X\map Y$ such that $dist_N (\phi(x), \iota_\e(x)) \leq \e$ for any $x \in X$, and that $\|\iota_\e\|_{{\rm Lip}(X,Y)} \leq (1+\e)\|\phi\|_{{\rm Lip}(X,Y)}$.

\section{Proof}

This section is devoted to the proof of Theorem \ref{thm: main}. We shall establish the more general
\begin{theorem}\label{thm: general}
Let $X \subset \R^\nu$ be a locally Ahlfors $n$-regular, locally linearly contractible homological $n$-manifold; $\nu \geq n \geq 2$. Assume the existence of a Cartan--Whitney presentation in $W^{1,2}_{\rm loc}(X)$. Let $F: X \map \R^{\nu'}$ be a Lipschitz map, $\nu' \geq n$. Then the image $F(X)$ is a smoothable topological $n$-manifold if and only if $\sing\,F = \emptyset$. \end{theorem}

\begin{proof}[Proof for  Theorem \ref{thm: general} $\Rightarrow$ Theorem \ref{thm: main}]
Let $M$ be an $n$-dimensional compact Riemannian manifold. It admits a $C^k$-isometric embedding into $\R^\nu$, for $k \geq 3$ and $\nu$ large enough, by Nash's theorem \cite{nash}. Cover $M$ by finitely many coordinate charts and fix one such chart $U$. Denote by $\{\nicefrac{\p}{\p x^1},\ldots,\nicefrac{\p}{\p x^n} \} \subset \G(TU)$ a local orthonormal frame on $U$, and let $\{dx^1,\ldots,dx^n\} \subset \G(T^*U)$ be the corresponding coframe. Clearly $\{dx^1,\ldots,dx^n\}$  are flat $1$-forms with $\sup_{1\leq i \leq n}\|dx^i\|_{L^\infty(U)} \leq C<\infty$ and $d(dx^i)=0$. Moreover, since $dx^1 \wedge\ldots\wedge dx^n$ is a volume $n$-form, hence Eq.\,\eqref{rho} is verified. Thus, $\{dx^1,\ldots,dx^n\}$ constitute a Cartan--Whitney presentation in $C^\infty(U) \subset W^{1,2}_{\rm loc}(M)$.

Furthermore, since $F: M \map N$ is Lipschitz and $F(U)$ is precompact in the manifold $N$, by shrinking $U$ if necessary, we can take $F(U)$ lying in one single geodesic normal ball $\mathscr{B}$ on $N$. As the exponential map on $\mathscr{B}$ is a $C^\infty$-diffeomorphism, by composing with it we may assume that $F$ maps into the Euclidean space $\R^{\nu'}$.

Finally, by the local nature of the statement of Theorem \ref{thm: main}, it remains to show that  {\em under the assumption that $F(U)$ satisfies \eqref{X}}, $F|U$ is approximable by smooth immersions if and only if $F(U)$ is smoothable. For the forward implication, we may utilise verbatim the arguments on p.32 in \cite{hk}; in particular, the proof of Eqs.\,(8.9) and (8.10) and an application of the results in \cite{michael, s1}. For the converse, one may pass to a sub-atlas of the $C^\infty$-structure and take $\iota_\e \equiv F$. Thus Theorem \ref{thm: main} follows.   \end{proof}

We are now ready to show Theorem \ref{thm: general}. Heuristically, the key idea is that $\sing\,F=\emptyset$ prevents $F$ from pinching necks.

\begin{proof}[Proof of Theorem \ref{thm: general}]

Fix an open neighbourhood $U \subset X$, on which there is a given Cartan--Whitney presentation $\rho=(\rho_1, \ldots, \rho_n) \in W^{1,2}(U)$. Let $F: U \map \R^{\nu'}$ be a Lipschitz map. Again, by the local nature of the statement, we may assume that $F(U)$ is orientable and show that $F(U)$ is smoothable if and only if $\sing\,(F|U)=\emptyset$. In the sequel we view $F$ as defined on $U$.

To this end, consider the pushforward $n$-tuple of flat $1$-forms:
\begin{equation*}
F_\# \rho = \big( F_\#\rho_1, \ldots, F_\# \rho_n \big).
\end{equation*}
In view of Theorem \ref{thm: hk} (2), it then suffices to prove the equivalence between $\sing\,F=\emptyset$ and the following two conditions altogether: $F_\# \rho$ defines a measurable cotangent structure \textit{\`{a} la} Sullivan \cite{s1, s2, s3} on $F(U)$, and that
\begin{equation}\label{reg}
F_\#\rho \in W^{1,2}_{\rm loc}(\R^{\nu'}).
\end{equation}

The Sobolev regularity condition \eqref{reg} is automatic, as $W^{1,2}$-tensors are preserved under pushforward via Lipschitz functions. In the sequel, we show that $\sing\,F = \emptyset$ if and only if $F_\#\rho$ yields a measurable cotangent structure, momentarily assuming that $F(U)$ satisfies \eqref{X}.

Let us first suppose $\sing\,F=\emptyset$ and deduce that $F_\#\rho$ induces a measurable cotangent structure. By definition, for each $x \in U$, every element of the generalised differential $\p F(x)$ is of maximal rank. Denote by $\mathscr{E}$ the set of points on $U$ where $dF$ do not exist; $\mathcal{H}^n(\mathscr{E})=0$ by Rademacher's theorem. In addition, clearly a necessary condition for $\sing\,F=\emptyset$ is that the differential $dF: TU \map T\R^{\nu'}$ (defined in the distributional sense; see $\S \ref{sec: prelim}$) is invertible at $\mathcal{H}^n$-{\it a.e.} point on $U$.

Our goal is to show that 
	\begin{equation}\label{ess 2}
	 \essinf_{F(U)} \,\,\star F_\#\big(\rho_1\wedge\ldots\wedge \rho_n\big) \geq c_1 >0
	\end{equation}
under the assumption:
\begin{equation}\label{ess 1}
 \essinf_U \,\,\star\big( \rho_1 \wedge \ldots\wedge \rho_n\big) \geq c_0 >0.
	\end{equation}
Suppose \eqref{ess 2} were false. Then, for any $\e >0$ there would be a set $\Sigma \subset U$ with $\mathcal{H}^n(\Sigma) >0$ and $\star F_\#(\rho_1\wedge \ldots\wedge \rho_n) < \e$ on $\Sigma$. Without loss of generality we may take $\Sigma$ to be the metric ball $B(p,2r) \subset U $ with an $\mathcal{H}^n$-null set $\G$ removed, such that $\G \supset B(p,2r) \cap \mathscr{E}$. After passing to subsequences if needed, one can find a convergent sequence of points $\{q_i\} \subset \Sigma \sim \G$ such that $q_j \map q \in \overline{B(p,r)}$ and that 
\begin{equation}\label{small det}
\det \big(dF(q_j)\big) < \nicefrac{\e}{c_0}, 
\end{equation} 
where $c_0$ is as in \eqref{ess 1}. Indeed, observe the identity $$\star F_\#(dx^1\wedge \ldots\wedge dx^n) = \det\,dF \big\{ \star \big( \rho_1\wedge\ldots\wedge \rho_n \big) \big\}$$ 
wherever $dF$ is invertible;  the determinant is well-defined as $F$ is Lipschitz. Thus, in view of Eq.\,\eqref{ess 1}, we have
\begin{equation*}
\star F_\#(dx^1\wedge \ldots\wedge dx^n)  \geq c_0\, \det\, dF
\end{equation*}
outside an $\mathcal{H}^n$-null set on $U$. Thus Eq.\eqref{small det} follows.  However, in the limits of $\e \searrow 0$ and $j=j(\e) \nearrow \infty$, Eqs.\,\eqref{ess 1}, \eqref{small det} and the rank-nullity theorem imply (via a diagonalisation argument) that  $dF$ cannot be of maximal rank at the limiting point $q$. Thus $q \in \sing\,F$, which yields a contradiction.
	
	Conversely, assume that $F$ is not approximable by smooth immersions; we shall find a point $q'$ in $\sing\,F$. By definition, it suffices to find a sequence $\{q_j'\} \subset U \sim \mathscr{E}$ such that $q_j' \map q'$ and that $\lim_{j\nearrow\infty} dF(q_j')$ has rank less than $n$. Indeed, again due to Theorem \ref{thm: hk}, the non-approxmability of $F$ leads to
	\begin{equation}\label{ess 3}
	\essinf_{F(U)}\,\, \star F_\# \big(\rho_1 \wedge \ldots \wedge \rho_n\big) = 0.
	\end{equation}
In view of the Lipschitzness of $F$ and the precompactness of $U$, the essential infimum in Eq.\,\eqref{ess 3} is attained at a point. That is, for some $q' \in U$ there holds $\star F_\#(\rho_1\wedge\ldots\wedge \rho_n)\big(F(q')\big)=0$. Moreover, for any $\e>0$, there exists a sequence $\{q_j'\} \subset U \sim \mathscr{E}$ depending possibly on $\e$, such that $q_j' \map q'$  (hence $F(q_j')\map F(q')$) and  that  $$\star F_\#(\rho_1\wedge\ldots\wedge \rho_n)\big(F(q_j')\big) < \e.$$ Since $(\rho_1,\ldots,\rho_n)$ is a Cartan--Whitney presentation on $U$, by Eq.\,\eqref{ess 1} we may assume that all the $q_j'$ chosen above satisfy
\begin{equation*}
\star \big(\rho_1\wedge\ldots\wedge \rho_n\big)(q_j') \geq c_0 >0.
\end{equation*}
By an analogous estimate as for Eq.\,\eqref{det bounds}, we can again bound
\begin{equation*}
\det\, dF(q_j')\leq \nicefrac{\e}{c_0}.
\end{equation*}
In particular, the determinant of $dF$ are well-defined at each point $q_j'$. By sending $\e \searrow 0$, $j=j(\e) \nearrow \infty$ and using a standard diagonalisation argument, we find that any pointwise subsequential limit $\mathfrak{m}$ of $dF(q'_j)$ verifies
$${\rm rank}\,\,\mathfrak{m} \leq n-1.$$ 
So, as $\{q_j'\}$ converges to $q'$, we immediately obtain that $q' \in \sing\, F$.

We are now left to prove that $F(U)$ satisfies the structural assumptions in \eqref{X}. For this purpose, we shall make crucial use of Eq.\,\eqref{ess 2} established above. Indeed, by the Lipschitzness of $F$ let us rewrite Eqs.\,\eqref{ess 1} and \eqref{ess 2} as
\begin{eqnarray}
&c_0 \leq \essinf_{U} \star \big(\rho_1\wedge\ldots\wedge \rho_n\big) \leq  \esssup_{U} \star \big(\rho_1\wedge\ldots\wedge \rho_n\big) \leq C_0,\label{ess 1'}\\
&c_1 \leq \essinf_{U} \star F_\#\big(\rho_1\wedge\ldots\wedge \rho_n\big) \leq  \esssup_{U} \star F_\# \big(\rho_1\wedge\ldots\wedge \rho_n\big) \leq C_1.\label{ess 2'}
\end{eqnarray}
where $C_0$ depends on the flat norm of $\rho$, and $C_1$ additionally on the Lipschitz norm of $F$.

Indeed, Eqs.\,\eqref{ess 1'}\eqref{ess 2'} imply that
\begin{equation}\label{det bounds}
0< \lambda := \frac{c_1}{C_0} \leq \big|\det (dF)\big| \leq \frac{C_1}{c_0} =:\Lambda <\infty \qquad \mathcal{H}^n-\text{a.e. on } U.
\end{equation}
As a result, given any metric ball $\widehat{B}(x,r)\subset F(U)$, we can find radii $0<r_-<r_+$ such that the following inclusions hold {\em outside at most a $\mathcal{H}^n$-null set}:
\begin{equation}\label{inclusion}
F\big(B(x, r_-)\big) \subset \widehat{B}(x,r) \subset F\big(B(x, r_+) \big).
\end{equation}
Here and hereafter, we always use $\widehat{B}(\bullet, \bullet)$ to denote the metric balls in $F(U)$; the notation $B(\bullet, \bullet)$ is reserved for the metric balls in $U$. 

To proceed, notice that one may take $r_\pm$ to be equal to $r$ modulo a multiplicative factor depending only on $\Lambda$ and $\lambda$. Thus, by the area formula (\cite{f}, 3.2.20), there exists a constant $0<c_3<\infty$ depending only on $\Lambda, \lambda$, the Lipschitz norm of $F$ and the local data of $X$ such that $c_3 r^n\leq \mathcal{H}^n(B(x,r))\leq c_3r^n$. This gives the local $n$-Ahlfors regularity of $F(U)$. The local linear contractibility of $F(U)$ follows similarly from \eqref{inclusion}.

Finally, Let $\mathscr{B} \subset U$ be the set on which  Eq.\,\eqref{det bounds} fails. Then Eq.\,\eqref{hmlg mfd} clearly holds on $U \sim \mathscr{B}$, since $F$ is a homeomorphism thereon and hence leaves $H_\bullet(X,X\sim\{\ast\};\mathbb{Z})$ invariant. For $\ast \in \mathscr{B}$, by excision we have
\begin{equation*}
H_\bullet \Big(F(U) \cup \{\ast\} \sim \mathscr{B}, F(U) \sim \mathscr{B};\mathbb{Z}\Big) \cong H_\bullet\Big( F(U), F(U) \sim \{\ast\};\mathbb{Z}\Big).
\end{equation*}
Utilising the facts that $F$ is Lipschitz (hence continuous) on $U$, that $F$ is a bi-Lipschitz homeomorphism onto $F(U) \sim \mathscr{B}$ by Eq.\,\eqref{det bounds}, and that $U$ is a Lipschitz $n$-manifold modulo reparametrisations (thanks to Theorem \ref{thm: hk} (1)), we deduce that $F(U)\cup\{\ast\}\sim \mathscr{B}$ deformation retracts onto $F(U) \sim \mathscr{B}$. This allows us to compute the relative homology from the reduced homology $\widetilde{H}_\bullet$:
\begin{equation*}
H_\bullet \Big(F(U) \cup \{\ast\} \sim \mathscr{B}, F(U) \sim \mathscr{B};\mathbb{Z}\Big) \cong \widetilde{H}_\bullet\Bigg(\frac{F(U) \cup \{\ast\}\sim \mathscr{B}}{ F(U)\sim\mathscr{B}};\mathbb{Z}\Bigg) \cong  H_\bullet \big(\R^n; \R^n\sim\{0\};\mathbb{Z}\big).\end{equation*}
Hence $F(U)$ is a homology $n$-manifold.

The proof is now complete.   \end{proof}

\bigskip
\noindent
{\bf Acknowledgement}.
This work has been done during Siran Li's stay as a CRM--ISM postdoctoral fellow at Centre de Recherches Math\'{e}matiques, Universit\'{e} de Montr\'{e}al and Institut des Sciences Math\'{e}matiques. The author would like to thank these institutions for their hospitality.

\end{document}